\documentclass[final, 12pt]{amsart}
\newif\ifsharetheoremcounter\sharetheoremcountertrue

\usepackage{amsmath, amssymb, esint}
\usepackage{graphicx, xcolor}
\usepackage{microtype}

\ifdefined\ifenablehyperref\else\expandafter\newif\csname
        ifenablehyperref\endcsname\enablehyperreftrue\fi
\ifenablehyperref\usepackage[final]{hyperref}\else\fi

\ifdefined\ifdraft\else\expandafter\newif\csname ifdraft\endcsname\fi
\ifdim\overfullrule>0pt \drafttrue\else\draftfalse\fi

\ifdefined\ifsharetheoremcounter\else\expandafter\newif\csname
        ifsharetheoremcounter\endcsname\sharetheoremcounterfalse\fi

\allowdisplaybreaks[4]

\newcounter{modifiedcounter}

\ifdraft
    \newcommand{\fixme}[1]{\marginpar{\parbox{0in}{\fbox{\parbox{0.6255in}{
            \raggedright \scriptsize #1}}}}}
    
    \newcommand{\dummyline}[1]{\vspace{0.5em}\noindent\rule{\textwidth}{#1 pt}
            \vspace{0.5em}}

\else
    \newcommand{\fixme}[1]{}
    
    \newcommand{\dummyline}[1]{}

\fi

\newlength{\defaultarrayrulewidth}
\theoremstyle{remark}
\numberwithin{equation}{section}

\newtheorem*{claim*}{Claim}

\newtheorem{theorem}{Theorem}[section]

\ifsharetheoremcounter
    \newtheorem{definition}[theorem]{Definition}
    \newtheorem{claim}[theorem]{Claim}
    \newtheorem{lemma}[theorem]{Lemma}
    \newtheorem{proposition}[theorem]{Proposition}
    \newtheorem{corollary}[theorem]{Corollary}
    \newtheorem{remark}[theorem]{Remark}
\else

    \newtheorem{lemma}{Lemma}[section]
    \newtheorem{proposition}{Proposition}[section]
    \newtheorem{corollary}{Corollary}[section]
    \newtheorem{remark}{Remark}[section]
\fi

\DeclareMathOperator{\grad}{grad}

\DeclareMathOperator{\id}{id}

\DeclareMathOperator{\Ric}{Ric}

\DeclareMathOperator{\tr}{tr}

\DeclareMathOperator*{\osc}{osc}

\author[S. Huang]{Shaochuang Huang}
\address[Shaochuang Huang]{School of Science, Shenzhen Campus of Sun Yat-sen
        University, No. 66, Gongchang Road, Guangming District, Shenzhen,
        Guangdong 518107, P. R. China.}
\email{huangshch23@mail.sysu.edu.cn}
\thanks{S. Huang is partially supported by a regular fund from Shenzhen Science
        and Technology Program No. JCYJ20240813151005007 and a start-up
        fund from SYSU}

\author[Z. Peng]{Zhuo Peng}
\address[Zhuo Peng]{School of Science, Shenzhen Campus of Sun Yat-sen
        University, No. 66, Gongchang Road, Guangming District, Shenzhen,
        Guangdong 518107, P. R. China.}
\email{pengzh55@mail2.sysu.edu.cn}

\usepackage{lmodern}
\usepackage{needspace}
\hypersetup{hidelinks,
        pdftitle={Strong Uniqueness for Instantaneously Complete Chern-Ricci
                Flow},
        pdfsubject={Instantaneously complete Chern-Ricci and Kahler-Ricci
                flows},
        pdfkeywords={Chern-Ricci flow, strong uniqueness, Hermitian surface}}

\begin{document}

\title[Strong uniqueness for Chern-Ricci flow]
        {Strong Uniqueness for Instantaneously Complete Chern-Ricci Flow}
\subjclass[2020]{Primary 53E30; Secondary 53C55, 53E20}
\keywords{Chern-Ricci flow, K\"ahler-Ricci flow, Ricci flow, strong uniqueness,
        instantaneous completeness}
\date{}

\begin{abstract}
    In this paper, we study strong uniqueness for smooth instantaneously
            complete Chern-Ricci flows on noncompact complex manifolds.
    Under suitable reference metric and positivity assumptions, we establish
            strong uniqueness for initial Hermitian metrics which are K\"ahler
            outside a compact set.
    The initial metric may be incomplete, and no curvature bounds or metric equivalence assumptions are imposed on the evolving solutions.
    In particular, we obtain strong uniqueness from complete initial metrics in
            this class with bounded Chern-Ricci curvature.
    We also prove stationarity of flows from the Euclidean metric on
            \(\mathbb C^n\) and uniqueness of the instantaneously complete flow
            from the Euclidean metric on the complex unit ball for all positive
            time.
    Finally, on complex surfaces, we show that a Ricci flow remaining Hermitian
            for a fixed complex structure is K\"ahler at all times if it is
            K\"ahler at one time, and derive a corresponding strong uniqueness result
            for Ricci flows in this class.
\end{abstract}

\maketitle

\section{Introduction}

Uniqueness for Ricci flow is an important and interesting problem in the study
        of Ricci flow.
In 1982, Hamilton introduced Ricci flow to study the geometry and topology of
        manifolds in his seminal work \cite{Hamilton82JDG}, where he proved the
        short-time existence and uniqueness for Ricci flow on closed manifolds
        via the sophisticated machinery of the Nash-Moser implicit function
        theorem.
Later, Hamilton \cite{Hamilton95SDG} provided an elegant proof of uniqueness
        for Ricci flow on closed manifolds using DeTurck's trick and harmonic
        map heat flows.
On a noncompact manifold, uniqueness usually requires some control
        of the solutions at infinity.
Chen and Zhu \cite{CZ06JDGb} proved uniqueness among complete solutions with
        uniformly bounded curvature by considering harmonic map heat flows
        coupled with Ricci flow with a fixed target.
Kotschwar \cite{Kotschwar14CAG} obtained uniqueness via energy method under a
        time-integrable curvature bound, see also
        \cite{Kotschwar17MRL,HT22JGA}.
Recently, Lee \cite{Lee2503} proved uniqueness for complete solutions
        satisfying a scaling-invariant $C/t$ curvature bound by considering
        harmonic map heat flows coupled with Ricci flow with evolving target
        metrics.
Very recently, Flower \cite{Flower2610} obtained uniqueness for Ricci flow
        starting from certain Alexandrov spaces under scaling-invariant
        curvature and injectivity radius bounds.

On the other hand, Chen \cite{Chen09JDG} proved strong uniqueness results for
        complete Ricci flows in dimensions $2$ and $3$.
Here strong uniqueness means uniqueness without a priori curvature or metric
        bounds on the evolving solutions.
In particular, Chen proved that the Euclidean metric on $\mathbb R^2$ and
        $\mathbb R^3$ is stationary along complete Ricci flow.
Strong uniqueness is a distinctive rigidity phenomenon in Ricci flow.
By contrast, even the classical heat equation on the real line admits a
        nontrivial smooth solution with zero initial data in the absence of
        growth restrictions at spatial infinity, see the famous Tychonoff's
        example \cite{Tychonoff35RecMathMoscou}.
In dimension $2$, Topping \cite{Topping15GT} established strong uniqueness
        among instantaneously complete Ricci flows from arbitrary initial
        metrics, without any completeness and curvature bound assumptions on
        initial data.
Cheng and Zhang \cite{CZ25AnnPDE} proved strong uniqueness from Euclidean
        initial data under the assumption of $\mathcal{L}$-completeness for the
        Ricci flow in higher dimensions.

In the K\"ahler case, the potential equation provides another approach to the
        uniqueness problem.
Fan \cite{Fan07PAMS} proved a uniqueness theorem for complete K\"ahler-Ricci
        flow under uniform equivalence to the initial metric, assuming bounded
        initial curvature and a global Ricci potential.
The uniqueness theorem of Chau, Li and Tam \cite[Theorem~2.1]{CLT17TAMS}
        includes flows uniformly equivalent to a complete reference metric
        admitting a proper exhaustion with bounded gradient and complex
        Hessian.
Lee \cite[Theorem~1.1]{Lee19JGA} proved a uniqueness theorem for complete
        K\"ahler-Ricci flow assuming uniform equivalence of the two flows and a
        $C/t$ curvature bound on one of them.
Chau, Li and Shen \cite[Proposition~3.1]{CLS18AdvMath} proved uniqueness in a
        class of bounded potentials along the flows on complements of smooth
        divisors in compact K\"ahler manifolds.

In this paper, we will consider the strong uniqueness problem for a complex
        analogue of Ricci flow, namely the Chern-Ricci flow on a 
        complex manifold $(M^n,J)$ of complex dimension $n$, and we require no
        curvature or metric equivalence assumption on the evolving
        solutions.

The Chern-Ricci flow was first studied by Gill \cite{Gill11CAG} in a special
        setting and then introduced and developed in general by Tosatti and
        Weinkove in \cite{TW15JDG} for closed complex manifolds.
Lee and Tam \cite{LT20JDG} studied the uniformization conjecture in K\"ahler
        geometry by developing Chern-Ricci flow theory on complete noncompact
        manifolds.
In \cite{HLT19CVPDE}, the first-named author in this paper, Lee and Tam studied
        the existence problem for instantaneously complete Chern-Ricci flow on
        noncompact complex manifolds.
One motivation for this paper is to investigate strong uniqueness among such
        instantaneously complete Chern-Ricci flows.
Since any Ricci flow on an oriented $2$-dimensional Riemannian manifold is
        K\"ahler, our results in this paper may be viewed as higher dimensional
        analogues of Topping's strong uniqueness for Ricci flow on surfaces
        \cite{Topping15GT}.

Before stating our main results, we fix some notation and conventions.
For a Hermitian metric $h$, we denote its associated real $(1,1)$-form by
        $\omega$ and write
\[
    \omega=\sqrt{-1}\,h_{i\bar j}\,dz^i\wedge d\bar z^j,
    \qquad
    \Ric(\omega):=-\sqrt{-1}\,\partial\bar\partial
            \log\det(h_{i\bar j}).
\]
Here $\partial_i:=\partial/\partial z^i$ and
        $\partial_{\bar j}:=\partial/\partial\bar z^j$.
The form $\Ric(\omega)$ is the first Chern-Ricci form.
We denote by $\Ric_g^{\mathbb R}$ and $\mathcal R_g^{\mathbb R}$ the
        Riemannian Ricci tensor and scalar curvature of $g$, respectively,
        omitting the metric subscript when clear.
We also use $h$ for the underlying real inner product, so that
        $\omega(X,Y)=h(JX,Y)$ and
        $h_{i\bar j}=h(\partial_i,\partial_{\bar j})$, with $h$ extended
        complex bilinearly to the complexified tangent bundle.
Thus $\omega^n/n!$ is the volume form of $h$.
For a reference metric $\hat h$, we denote its associated form by $\hat\omega$,
        and use corresponding notation for other metrics and their associated
        forms.
When no confusion can arise, we also refer to $\omega$ as a Hermitian metric.    Completeness and curvature conditions then refer to the associated
        metric $h$.
We use form notation for exterior powers and metric notation for norms,
        connections, and differential operators.
We denote the Chern connection by $\nabla$ and the Levi-Civita connection by
        $D$, adding a metric superscript when needed.
We write $\grad_g f$ for the real gradient of a function $f$ with respect to a
        metric $g$.
For a Hermitian metric $h$ and a smooth real function $f$, we use
\[
    |\partial f|_h^2=h^{i\bar j}\partial_i f\,\partial_{\bar j}f
            =\frac12|\grad_h f|_h^2,
\]
        where $\grad_h f$ is taken with respect to the underlying real metric.
The Chern-Ricci flow is
\begin{equation}
    \label{eqn chern flow}
    \partial_t\omega=-\Ric(\omega),\qquad \omega(0)=\omega_0.
\end{equation}
Since the Chern-Ricci form is closed,
\begin{equation}
    \label{eqn preservation closedness}
    d\omega(t)=d\omega_0.
\end{equation}
In particular, K\"ahler initial metrics remain K\"ahler, and
        \eqref{eqn chern flow} then becomes the K\"ahler-Ricci flow.
See \cite{TW22Lincei} for more systematic discussions on Chern-Ricci flow.
A solution is called instantaneously complete if its associated metric $h(t)$
        is complete for every $t>0$.
Throughout, manifolds are smooth, connected and without boundary. Metrics are smooth.
        Smoothness of flows on $M\times[0,S]$ includes smoothness up to $t=0$
        on compact subsets of $M$.

As mentioned above, the existence results in \cite{HLT19CVPDE} lead naturally
        to the question of uniqueness among all smooth instantaneously complete
        solutions.
Our theorem addresses this question under the reference and initial assumptions below. There are no any requirements on the evolving metrics except smoothness and
        instantaneous completeness.

\begin{theorem}
    \label{thm main}
    Let $\hat\omega$ be a complete Hermitian metric on an $n$-dimensional
            complex manifold $M$, K\"ahler outside a compact set, such that
    \begin{equation}
        \label{eqn reference curvature bound}
        \Ric(\hat\omega)\ge-K\hat\omega
    \end{equation}
            on $M$ for some $K\ge0$.
    Let $\omega_0$ be a Hermitian metric on $M$, K\"ahler outside a compact
            set, and satisfying
    \begin{equation}
        \label{eqn reference bounds}
        \omega_0\le C\hat\omega
    \end{equation}
            on $M$ for some $C\ge1$.
    Assume that there exist $s>0$, $\beta>0$, and a real function
            $u\in C^\infty(M)\cap L^\infty(M)$ such that
    \begin{equation}
        \label{eqn reference positivity}
        \omega_0-s\Ric(\hat\omega)+\sqrt{-1}\partial\bar\partial u
                \ge \beta\hat\omega.
    \end{equation}
    Suppose that $\omega_1(t)$ and $\omega_2(t)$ are smooth instantaneously
            complete solutions to the Chern-Ricci flow \eqref{eqn chern flow}
            on $M\times[0,S]$ with the same initial data $\omega_0$.
    Then
    \[
        \omega_1(t)\equiv\omega_2(t)
                \quad\text{for all}\quad 0\le t\le\min\{S,s\}.
    \]
\end{theorem}

Inequalities between real $(1,1)$-forms are understood in the Hermitian sense.
All the assumptions in
        \eqref{eqn reference curvature bound}-\eqref{eqn reference positivity}
        concern the initial metric and the reference data.
The initial metric need not be complete or have bounded curvature, and no
        derivative bound on $u$ is assumed.
The K\"ahler condition is required only outside a compact set.
By \eqref{eqn preservation closedness}, this condition is preserved by the
        flow.
The positivity condition \eqref{eqn reference positivity} is analogous to those
        used in the existence theory in \cite{LT20JDG,HLT19CVPDE}.
Although Theorems 1.1 and 1.2 in \cite{HLT19CVPDE} are stated for a normalized
        Chern–Ricci flow, their unnormalized counterparts follow by a standard
        rescaling and time change, see also \cite{CL20TAMS}.

After the first version of this manuscript appeared on arXiv, the authors
        proved  strong uniqueness for smooth complete Riemannian
        Ricci flows from complete initial metrics with Ricci curvature bounded
        below and uniformly positive injectivity radius
        \cite[Corollary~1.2]{HP2610}.
For K\"ahler initial metrics satisfying these geometric bounds, that result
        requires no K\"ahler or Hermitian condition on the evolving solutions.
Theorem~\ref{thm main} considers the initial and reference assumptions above, allowing
        incomplete initial metrics and imposing no injectivity radius bound.

We briefly describe the proof here. The upper bound for the potential uses only the complete reference
        metric $\hat\omega$ and
        \eqref{eqn reference curvature bound}-\eqref{eqn reference bounds}.
By a trace inequality, Laplacian comparison for $\hat\omega$, and a cutoff
argument, we obtain the upper bound for the potential by a maximum principle argument.
This does not require completeness of the flow or a priori bounds on the potential. For the lower bound, we use \eqref{eqn reference positivity} and the scalar
curvature estimate to derive a quadratic differential inequality wherever
a suitable modification of the potential is negative.
By applying the maximum principle from time $\varepsilon>0$ and letting
$\varepsilon\to0^+$, we obtain the lower bound.
The K\"ahler condition outside a compact set allows us to apply this
maximum principle on closed positive-time intervals. Then the potentials are bounded and converge uniformly to zero
as $t\to0^+$.
By applying the maximum principle again to their differences 
with respect to the corresponding flows,  we obtain uniqueness.

By taking the initial metric as the reference metric in
        Theorem~\ref{thm main}, we obtain the following corollary.

\begin{corollary}
    \label{cor complete initial}
    Let $\omega_0$ be a complete Hermitian metric on an $n$-dimensional complex
            manifold $M$, K\"ahler outside a compact set, and satisfying
    \begin{equation}
        \label{eqn initial ricci bounds}
        -K^-\omega_0\le\Ric(\omega_0)\le K^+\omega_0\quad\text{on}\quad M
                \quad\text{for some}\quad K^-,K^+\ge0.
    \end{equation}
    Then any two smooth instantaneously complete Chern-Ricci flows with initial
            data $\omega_0$ agree at every common time $t$ for which $K^+t<1$.
\end{corollary}

When $K^+>0$, equality also holds at $t=1/K^+$ if both solutions are defined
        there by continuity.
Corollary~\ref{cor complete initial} includes complete K\"ahler initial metrics
        with bounded Ricci curvature, without a bound on the full curvature
        tensor.
In a related result, Lee \cite[Corollary~4.1]{Lee2503} proves strong
        uniqueness among complete $U(n)$-invariant K\"ahler-Ricci flows on
        $\mathbb C^n$ when the initial metric has nonnegative bisectional
        curvature and uniformly noncollapsed unit balls.
Lee's argument obtains the scaling-invariant curvature estimate needed for
        uniqueness from these geometric assumptions.

In particular, by taking $K^-=K^+=0$ in
        Corollary~\ref{cor complete initial} and comparing with the stationary
        Euclidean flow, we obtain the following corollary. For this special case, we provide a direct and simple proof in the appendix. 
It  provides a partial answer to Chen's strong uniqueness question
        \cite[Section~4]{Chen09JDG}, which is another motivation for this
        work. As mentioned above,
the authors \cite[Corollary~1.4]{HP2610} resolved Chen's strong uniqueness question in every dimension very recently.

\begin{corollary}
    Let
    \[
        \omega_{\mathrm E}:=\frac{\sqrt{-1}}2
                \sum_{j=1}^n dz^j\wedge d\bar z^j
    \]
            be the standard Euclidean metric on $\mathbb C^n$.
    Then every smooth instantaneously complete Chern-Ricci flow $\omega(t)$ on
            $\mathbb C^n\times[0,S]$ with $\omega(0)=\omega_{\mathrm E}$
            satisfies
    \[
        \omega(t)\equiv\omega_{\mathrm E}
                \quad\text{for all}\quad 0\le t\le S.
    \]
\end{corollary}

A second application concerns incomplete initial metrics on the unit ball.
Ge, Lin and Shen \cite{GLS19MRL} constructed instantaneously complete
        K\"ahler-Ricci flows on strictly pseudoconvex domains and asked whether
        they are unique, see \cite[Remark~3.7]{GLS19MRL}.
The following corollary gives strong uniqueness under the stated initial metric
        assumptions.
In particular, it holds for the Euclidean initial data.
The existence part of such smooth instantaneously complete Chern-Ricci flows
        starting from the Euclidean initial data follows from \cite{HLT19CVPDE}
        with a standard rescaling and time change, see also \cite{CL20TAMS}.
We include a short proof in Section~\ref{sec consequences}.

\begin{corollary}
    \label{cor unit ball}
    Let $\mathbb B^n:=\{z\in\mathbb C^n\mid |z|<1\}$, and let
    \[
        \hat\omega:=\sqrt{-1}\partial\bar\partial[-\log(1-|z|^2)]
    \] be the standard  hyperbolic metric on $\mathbb B^n$.
    Assume that $\omega_0$ is a Hermitian metric on $\mathbb B^n$ such that
            $\omega_0\le C\hat\omega$ for some $C>0$ and $d\omega_0$ has
            compact support.
    Then any two smooth instantaneously complete Chern–Ricci flows from
            $\omega_0$ agree throughout their common existence interval.
    In particular, the Euclidean form $\omega_{\mathrm E}$ restricted to
            $\mathbb B^n$ admits a unique such flow for all $t\ge0$.
\end{corollary}

Finally, we investigate the strong uniqueness for Riemannian Ricci flow on
        complex surfaces as an application of Theorem~\ref{thm main}.
We show that in real dimension $4$, namely complex dimension $2$, if a Ricci
        flow remains Hermitian for a fixed complex structure and is K\"ahler at
        one time, it is K\"ahler at every time.
This assertion requires neither completeness nor curvature bounds.
It follows from a pointwise evolution equation for the Lee form, which we derive in Section~\ref{sec hermitian}.
Theorem~\ref{thm main} then gives the following strong uniqueness within this
        class of Hermitian Ricci flows.
Preservation of K\"ahlerity along complete noncompact Ricci flows under
        curvature bounds was obtained by Shi \cite{Shi90PhD} and the
        first-named author in this paper and Tam \cite[Theorem~1.1]{HT18AJM}.
Here the assumption that the flow remains Hermitian is essential to the
        argument.

\begin{corollary}
    \label{cor hermitian uniqueness}
    Let $(M^2,J,g_0)$ be a  K\"ahler surface.
    Suppose there are reference data satisfying the hypotheses of
            Theorem~\ref{thm main} for its associated form $\omega_0$, with
            $n=2$ and some $s>0$.
    Let $g_1(t)$ and $g_2(t)$, $0\le t\le S$, be smooth solutions to the
            Riemannian Ricci flow, Hermitian for $J$ at every time and complete
            for $t>0$, with $g_1(0)=g_2(0)=g_0$.
    Then $g_1(t)\equiv g_2(t)$ for all $0\le t\le\min\{S,s/2\}$.
\end{corollary}

The paper is organized as follows.
In Section~\ref{sec maximum principle}, we prove some maximum principles which
        are useful for instantaneously complete Chern-Ricci flows which are
        K\"ahler outside a compact set.
In Section~\ref{sec potentials}, we establish the potential estimates and prove
        the main theorem.
Section~\ref{sec consequences} contains the applications.
In Section~\ref{sec hermitian}, we prove the preservation of K\"ahlerity and
        its consequences for Hermitian Ricci flows on complex surfaces.
In Appendix~\ref{sec ricci flat appendix}, we provide a direct proof of strong
        uniqueness for the K\"ahler Ricci-flat  case.

\medskip

{\bf AI disclosure statement.} In preparing this manuscript, the authors used
        ChatGPT-6 Astra as an interactive tool to assist with exposition and
        with exploring and checking some of the mathematical computations and
        arguments.
All interactions with AI were initiated and guided by the authors.
Any AI-generated text incorporated into the manuscript was substantially
        revised by the authors, who independently verified all mathematical
        statements and proofs and take full responsibility for the final
        content of the manuscript.

\section{Some maximum principles}
\label{sec maximum principle}

In this section, we prove some maximum principles by adapting the localization
        argument of Chen \cite[Proposition~2.1]{Chen09JDG}.
We allow a time-dependent coefficient in the quadratic term and keep track of
        the initial lower bound.
Related maximum principles were obtained in \cite[Lemma~5.1]{HLTT22CAG} and
        \cite[Proposition~2.3]{CL23MathAnn}.
We also allow a uniformly bounded error in the metric evolution and a uniformly
        bounded drift term in the parabolic operator.
We will apply these results to instantaneously complete Chern-Ricci flows which
        are K\"ahler outside a compact set.

Let $g(t)$ be a smooth family of Riemannian metrics and let $W(t)$ be a smooth
        time-dependent vector field.
We write
\[
    \Box:=\partial_t-\Delta_{g(t)}^{\mathbb R}
            -W(t)\cdot\grad_{g(t)}.
\]
For a continuous function $f$, we say that $\Box f\ge\delta(t)f^2$ holds in the
        barrier sense at $(x,t)$ if there is a smooth function $\bar f$ on a
        spacetime neighborhood of $(x,t)$ such that $\bar f\ge f$,
        $\bar f(x,t)=f(x,t)$, and
\[
    \Box\bar f(x,t)\ge\delta(t)f(x,t)^2.
\]
This condition is usually required only where $f<0$.
We use the same convention for $\Box f\ge0$.
We call $\bar f$ an upper barrier for $f$ at $(x,t)$.
For the reversed inequalities, we use the corresponding lower barrier
        functions.

\begin{proposition}
    \label{prop global maximum principle}
    Let $M$ be a manifold of real dimension $n$, and let $g(t)$, $t\in[s,S]$,
            be a smooth family of complete Riemannian metrics on $M$.
    Suppose that
    \[
        \partial_tg=-2\Ric_g^{\mathbb R}+E,
    \]
            where $E(t)$ is a smooth symmetric $2$-tensor.
    Assume that $E(t)$ and the smooth vector field $W(t)$ are uniformly bounded
            with respect to $g(t)$ on $M\times[s,S]$.
    Let $\delta:[s,S]\to(0,+\infty)$ be continuous, and suppose that
            $f\in C(M\times[s,S])$ satisfies $\Box f\ge\delta(t)f^2$ in the
            barrier sense whenever $f<0$, for all $s<t<S$.
    If $f(\cdot,s)\ge-K$ for some $K\in[0,+\infty]$, then
    \begin{equation}
        \label{eqn global initial bound}
        f(x,t)\ge-\frac{K}{1+K\int_s^t\delta(\tau)\,d\tau}
                \quad\text{for all}\quad x\in M\quad\text{and}\quad s<t\le S.
    \end{equation}
    When $K=+\infty$, the right-hand side is understood as
            $-(\int_s^t\delta(\tau)\,d\tau)^{-1}$, and no initial lower bound
            is required.
\end{proposition}

\begin{proof}
    Fix $p\in M$ and let $d_p(\cdot,t):=d_{g(t)}(\cdot,p)$.
    In the maximum principle calculations below, we may treat the distance and
            its cutoffs as smooth by Calabi's trick, see \cite[proof of
            Lemma~2.3]{HT18AJM} for example.
    By smoothness of $g(t)$, we may choose $r>0$ and $K_{\mathrm{loc}}>0$ such
            that $\Ric_{g(t)}^{\mathbb R}\le K_{\mathrm{loc}}g(t)$ on
            $B_{g(t)}(p,r)$ for all $t\in[s,S]$.
    We follow the proof of Perelman's distance estimate
            \cite[Lemma~8.3~(a)]{Perelman0211}, see also \cite[proof of
            Proposition~2.1]{Chen09JDG}.
    Let $\gamma:[0,\ell]\to M$ be a unit-speed minimizing $g(t)$-geodesic
            from $p$ to $x$, where $\ell:=d_p(x,t)>r$.
    By the first variation of length, we have
    \[
        \partial_t d_p
        =-\int_0^\ell\Ric^{\mathbb R}(\dot\gamma,\dot\gamma)\,d\sigma
                +\frac12\int_0^\ell E(\dot\gamma,\dot\gamma)\,d\sigma.
    \]
    By the standard index form estimate, we obtain
    \[
        \Delta_{g(t)}^{\mathbb R}d_p
        \le\frac{n-1}{r}
                -\int_0^\ell\Ric^{\mathbb R}(\dot\gamma,\dot\gamma)\,d\sigma
                +\frac23K_{\mathrm{loc}}r.
    \]
    Let $A:=\sup\limits_{M\times[s,S]}|E|_{g(t)}$ and
            $B:=\sup\limits_{M\times[s,S]}|W|_{g(t)}$.
    By the above estimates, we obtain
    \[
        \Box d_p\ge-c-B-\frac A2d_p,\qquad
                c:=\frac{n-1}{r}+\frac23K_{\mathrm{loc}}r.
    \]
    Consider
    \[
        \tilde d_p(x,t):=e^{A(t-s)/2}d_p(x,t)
                +(c+B)\int_s^t e^{A(\tau-s)/2}\,d\tau.
    \]
    By direct computation, we have $\Box\tilde d_p\ge0$ when $d_p>r$ and
            $|\grad_{g(t)}\tilde d_p|_{g(t)}\le e^{A(S-s)/2}$.
    Choose a smooth nonincreasing cutoff function $\eta:\mathbb R\to[0,1]$
            which equals $1$ on $(-\infty,7/8]$, vanishes on $[1,+\infty)$,
            and satisfies
    \[
        |\eta'|+|\eta''|+2\frac{(\eta')^2}{\eta}
                \le C\sqrt\eta\quad\text{whenever}\quad\eta>0,
    \]
            where $C>0$ is a universal constant, see
            \cite[Lemma~B.1]{DHP2605} for example.
    For all sufficiently large $R$, let
            $\varphi_R(x,t):=\eta(\tilde d_p(x,t)/R)$.
    Then $d_p>r$ wherever $0<\varphi_R<1$.
    By the estimates for $\tilde d_p$ and $\eta$, wherever $\varphi_R>0$,
            we obtain
    \[
        \Box\varphi_R+2\frac{|\grad_{g(t)}\varphi_R|_{g(t)}^2}{\varphi_R}
                \le\frac{C_1\sqrt{\varphi_R}}{R^2},\qquad
                C_1:=C e^{A(S-s)}.
    \]

    For later use in Proposition~\ref{prop initial value localization}, we
            present the remaining proof as follows.
    Suppose the smooth cutoff $0\le\varphi\le1$ has uniformly compact support
            and satisfies
    \[
        \Box\varphi+2\frac{|\grad_{g(t)}\varphi|_{g(t)}^2}{\varphi}
                \le\frac{C_2\sqrt\varphi}{\Lambda}
    \]
            whenever $\varphi>0$, for some constants $C_2,\Lambda>0$.
    Let $u:=\varphi f$.
    If $u\ge0$, then the desired lower bound holds trivially.
    Otherwise, at a negative spatial minimum of $u$, we have $\varphi>0$ and
    \begin{equation}
        \label{eqn box u}
        \Box u\ge\frac{\delta u^2}{\varphi}
                -\frac{C_2|u|}{\Lambda\sqrt\varphi}
    \end{equation}
            in the barrier sense.

    Let $u_{\min}(t):=\min\limits_M u(\cdot,t)$, which is continuous since
            $u$ has uniformly compact support.
    We use the lower left Dini derivative
    \[
        \frac{d_-}{dt}u_{\min}(t)
        :=\liminf_{h\to 0^+}
                \frac{u_{\min}(t)-u_{\min}(t-h)}{h},
    \]
            see \cite[Definition~10.22]{CGIKLN08Book}.
    At a negative minimum $(x,t)$, choose a smooth upper barrier $\bar u$ for
            $u$ satisfying \eqref{eqn box u}.
    It has a spatial minimum at $(x,t)$, and
    \[
        u_{\min}(t-h)\le u(x,t-h)\le\bar u(x,t-h),\qquad
                u_{\min}(t)=\bar u(x,t).
    \]
    Hence $(d_-/dt)u_{\min}\ge\partial_t\bar u(x,t)\ge\Box\bar u(x,t)$.
    For any fixed $\varepsilon\in(0,1)$, let
            $\kappa:=C_2/(\varepsilon\delta_*\Lambda)$, where
            $\delta_*:=\min\limits_{[s,S]}\delta>0$.
    By \eqref{eqn box u}, whenever $u_{\min}(t)\le-\kappa$, we obtain
    \begin{equation}
        \label{eqn dini umin}
        \frac{d_-}{dt}u_{\min}(t)
        \ge(1-\varepsilon)\delta(t)u_{\min}^2(t)\ge0.
    \end{equation}
    If $u_{\min}(t)<-\kappa$ for some $t\in(s,S)$, let
            $F(\tau):=u_{\min}(t-\tau)$.
    By applying \cite[Lemma~10.23]{CGIKLN08Book} with $c=F(0)$ on the interval
            where $F<-\kappa$, we obtain $F\le F(0)<-\kappa$.
    By continuity, this interval extends to $t-s$.
    Thus $u_{\min}(\tau)\le u_{\min}(t)<-\kappa$ for all
            $\tau\in[s,t]$, and hence \eqref{eqn dini umin} holds on $(s,t]$.
    It follows that
    \[
        -\frac1{u_{\min}(t)}\ge-\frac1{u_{\min}(s)}
                +(1-\varepsilon)\int_s^t\delta(\tau)\,d\tau.
    \]
    Therefore, we obtain
    \begin{equation}
        \label{eqn reusable cutoff comparison}
        \varphi f\ge\min\left\{
            -\frac{K}{1+(1-\varepsilon)K\int_s^t\delta(\tau)\,d\tau},
            -\frac{C_2}{\varepsilon\delta_*\Lambda}\right\}
                \quad\text{for all}\quad s<t<S.
    \end{equation}

    By applying this estimate to $\varphi_R$ with $C_2=C_1$ and $\Lambda=R^2$,
            and letting $R\to+\infty$ and then $\varepsilon\to 0^+$, we obtain
            \eqref{eqn global initial bound} for $s<t<S$.
    The conclusion at $t=S$ follows by continuity.
    This completes the proof.
\end{proof}

\begin{corollary}
    \label{cor bounded heat principle}
    Under the assumptions of Proposition~\ref{prop global maximum principle},
            suppose that $f\in C(M\times[s,S])$ is bounded from below,
            satisfies $\Box f\ge0$ in the barrier sense whenever $f<0$, and has
            $f(\cdot,s)\ge0$.
    Then $f\ge0$ on $M\times[s,S]$.
\end{corollary}

\begin{proof}
    Choose $K>0$ such that $f\ge-K$ and let $\tilde f:=e^{-(t-s)}f$.
    Then $\Box\tilde f\ge-\tilde f\ge\tilde f^2/K$ whenever $\tilde f<0$.
    Thus the result follows by applying
            Proposition~\ref{prop global maximum principle} to $\tilde f$
            with zero initial lower bound.
\end{proof}

We next apply these estimates to Chern-Ricci flow on a complex manifold of
        complex dimension $n$.
The complex Laplacian is
\[
    \Delta_{h(t)}f:=\tr_h(\sqrt{-1}\partial\bar\partial f)
            =h^{i\bar j}\partial_i\partial_{\bar j}f,
\]
        where $h(t)$ is the Hermitian metric associated to $\omega(t)$.
Let $g(t):=2h(t)$, viewed as a Riemannian metric.
If $h(t)$ is K\"ahler, then $\Delta_{h(t)}=\Delta_{g(t)}^{\mathbb R}$.

\begin{corollary}
    \label{cor chern maximum principle}
    Let $\omega(t)$ be a smooth instantaneously complete solution to the
            Chern-Ricci flow \eqref{eqn chern flow}, which is K\"ahler outside
            a compact set.
    On every $[s,\tau]\subset(0,S]$,
            Proposition~\ref{prop global maximum principle} and
            Corollary~\ref{cor bounded heat principle} hold for
            $\Box_h:=\partial_t-\Delta_{h(t)}$.
\end{corollary}

\begin{proof}
    For $g(t)=2h(t)$ as above, there are smooth tensors $E$ and vector fields
            $W$ such that
    \[
        \partial_tg=-2\Ric_g^{\mathbb R}+E,\qquad
        \Delta_{h(t)}=\Delta_{g(t)}^{\mathbb R}+W\cdot\grad_{g(t)}.
    \]
    Both $E$ and $W$ vanish where $\omega(t)$ is K\"ahler.
    Their norms are therefore bounded on $M\times[s,\tau]$, so
            Proposition~\ref{prop global maximum principle} and
            Corollary~\ref{cor bounded heat principle} hold for
            $\Box_h:=\partial_t-\Delta_{h(t)}$.
\end{proof}

We write $\mathcal R:=\tr_{h}\Ric(\omega)$ for the Chern scalar curvature.

\begin{corollary}
    \label{cor scalar lower bound}
    Under the assumptions of Corollary~\ref{cor chern maximum principle}, the
            Chern scalar curvature $\mathcal R$ satisfies
    \[
        \mathcal R(x,t)\ge-\frac nt
                \quad\text{for all}\quad x\in M\quad\text{and}\quad 0<t\le S.
    \]
    If $\omega_0$ is complete and $\mathcal R(\cdot,0)\ge-K$ for some $K\ge0$,
            then
    \[
        \mathcal R(x,t)\ge-\frac{K}{1+Kt/n}\quad\text{for all}\quad
                x\in M\quad\text{and}\quad 0<t\le S.
    \]
\end{corollary}

\begin{proof}
    If the flow is K\"ahler everywhere, then $g(t)=2h(t)$ is a Riemannian Ricci
            flow whose real scalar curvature satisfies
            $\mathcal R_{g(t)}^{\mathbb R}=\mathcal R$.
    In this case the bounds are Chen's estimate
            \cite[Corollary~2.3~(i)]{Chen09JDG}.
    For the general Chern-Ricci flow, by the evolution equation for the Chern
            scalar curvature, see \cite[Lemma~3.1]{GS15PAMS} for example, it
            holds that
    \[
        \Box_h\mathcal R=|\Ric(\omega)|_h^2
                \ge\frac1n\mathcal R^2.
    \]
    By applying Corollary~\ref{cor chern maximum principle} for $s>0$ and
            letting $s\to 0^+$, the first estimate follows.
    For the second estimate, since $\omega_0$ is complete, it follows from
            Proposition~\ref{prop global maximum principle} with $s=0$ and
            $\delta=1/n$.
\end{proof}

We also state a local counterpart of
        Proposition~\ref{prop global maximum principle} for Ricci flow,
        following the localization arguments of Chen
        \cite[Proposition~2.1]{Chen09JDG}, which may be of independent
        interest.

\begin{proposition}
    \label{prop initial value localization}
    Let $(M,g(t))$, $0\le t\le S$, be a Riemannian Ricci flow of real dimension
            $n\ge2$.
    Fix $p\in M$, $r>0$, and $A$ satisfying
    \[
        A\ge\frac{40}{3}(n-1)\frac{S}{r^2}+2.
    \]
    Suppose that, for every $t\in[0,S]$,
    \[
        B_t(p,Ar)\subset\subset M,\qquad
        \Ric_{g(t)}^{\mathbb R}\le(n-1)r^{-2}g(t)
                \quad\text{on}\quad B_t(p,r).
    \]
    Here we write $B_t(p,r):=B_{g(t)}(p,r)$.
    Let $\delta:[0,S]\to(0,+\infty)$ be continuous, and let
            $\delta_*:=\min\limits_{[0,S]}\delta>0$.
    Suppose $f\in C(M\times[0,S])$ satisfies
            $(\partial_t-\Delta_{g(t)}^{\mathbb R})f\ge\delta(t)f^2$ in the
            barrier sense whenever $f<0$ in $B_t(p,Ar)\times\{t\}$, for all
            $0<t<S$.
    If $f(\cdot,0)\ge-K$ on $B_0(p,Ar)$ for some $K\in[0,+\infty]$, then, for
            every $\varepsilon\in(0,1)$, there is $C_\varepsilon>0$
            depending only on $\varepsilon$ such that
    \begin{equation}
        \label{eqn local initial bound}
        f(x,t)\ge\min\left\{
            -\frac{K}{1+(1-\varepsilon)K\int_0^t\delta(\tau)\,d\tau},
            -\frac{C_\varepsilon}{\delta_*A^2r^2}\right\},
    \end{equation}
            for all $x\in B_t(p,3Ar/4)$ and all $0<t\le S$.
\end{proposition}

\begin{proof}
    Let $d_p(\cdot,t):=d_{g(t)}(\cdot,p)$, $c:=5(n-1)/(3r)$, and
            $\Box:=\partial_t-\Delta_{g(t)}^{\mathbb R}$.
    By Perelman's distance estimate \cite[Lemma~8.3(a)]{Perelman0211}, we have
    \[
        \Box d_p\ge-c\quad\text{whenever}\quad r<d_p<Ar.
    \]
    For the cutoff $\eta$ used in the proof of
            Proposition~\ref{prop global maximum principle}, consider
    \[
        \varphi(x,t):=\eta\left(\frac{d_p(x,t)+ct}{Ar}\right).
    \]
    Then $\varphi$ has uniformly compact support, equals one on $B_t(p,3Ar/4)$,
            and vanishes outside $B_t(p,Ar)$.
    Moreover, $r<d_p<Ar$ wherever $0<\varphi<1$.
    Consequently,
    \[
        \Box\varphi+2\frac{|\grad_{g(t)}\varphi|_{g(t)}^2}{\varphi}
                \le\frac{C_2\sqrt\varphi}{A^2r^2}
                \quad\text{whenever}\quad\varphi>0,
    \]
            for a universal constant $C_2>0$.
    By applying \eqref{eqn reusable cutoff comparison} from the proof of
            Proposition~\ref{prop global maximum principle} with $s=0$ and
            $\Lambda=A^2r^2$, we obtain \eqref{eqn local initial bound} for
            $0<t<S$, where $C_\varepsilon:=C_2/\varepsilon$.
    The case $t=S$ follows by continuity.
\end{proof}

\section{Potential estimates and strong uniqueness}
\label{sec potentials}

In this section, we first establish upper and lower bounds for the
        potential function.
We then prove Theorem~\ref{thm main} on strong
        uniqueness for instantaneously complete Chern-Ricci flows.
Here and in the remaining sections, $n$ again denotes the complex dimension.


Let $\omega(t)$ be a solution to the Chern-Ricci flow \eqref{eqn chern flow}
        and let $\hat\omega$ be a fixed reference Hermitian metric.
Define
\begin{equation}
    \label{eqn canonical potential}
    \psi(x,t):=\int_0^t\log
            \frac{\omega(\tau)^n}{\hat\omega^n}(x)\,d\tau
    \quad\text{and}\quad \alpha_t:=\omega_0-t\Ric(\hat\omega).
\end{equation}
By the flow equation \eqref{eqn chern flow} and the identity
\[
    \sqrt{-1}\partial\bar\partial
            \log\frac{\omega(t)^n}{\hat\omega^n}
    =-\Ric(\omega(t))+\Ric(\hat\omega),
\]
we obtain
\begin{equation}
    \label{eqn potential equation}
    \omega(t)=\alpha_t+\sqrt{-1}\partial\bar\partial\psi.
\end{equation}
By definition and the flow equation, we also have
\begin{equation}
    \label{eqn potential time derivatives}
    \partial_t\psi=\log\frac{\omega(t)^n}{\hat\omega^n}
            \quad\text{and}\quad\partial_t^2\psi=-\mathcal R.
\end{equation}

We first establish an upper bound for the potential using the fixed
        reference metric $\hat\omega$.
By a trace inequality, Laplacian comparison, and the maximum principle
        with spatial cutoffs, we obtain the estimate.
This argument requires neither completeness of the flow nor an
        a priori bound  on the potential.
For related upper bounds for the potential, see
        \cite{TW15JDG, HLT19CVPDE, LT20JDG, CL20TAMS, HLTT22CAG}.

\begin{lemma}
    \label{lem upper potential}
    Let $M$ be an $n$-dimensional complex manifold and $\hat\omega$ be a
            complete Hermitian metric on $M$ which is K\"ahler outside a
            compact set, and suppose that $\Ric(\hat\omega)\ge-K\hat\omega$ on
            $M$ for some $K\ge0$.
    Let $\omega(t)$ be a smooth Chern-Ricci flow on $M\times[0,S]$ with
            $\omega_0\le C\hat\omega$ for some $C\ge1$.
    Then the potential in \eqref{eqn canonical potential} satisfies
    \begin{equation*}
        \psi(x,t)\le n\int_0^t\log(C+K\tau)\,d\tau
                \quad\text{for all}\quad x\in M\quad\text{and}\quad 0\le t\le S.
    \end{equation*}
\end{lemma}

\begin{proof}
    Let $\hat h$ be the Hermitian metric associated with $\hat\omega$ and let
            $a(t):=C+Kt$.
    Since $\alpha_t\le a(t)\hat\omega$, by taking the $\hat h$-trace in
            \eqref{eqn potential equation} and applying the
            arithmetic-geometric mean inequality, we obtain
    \[
        \Delta_{\hat h}\psi
                =\tr_{\hat h}\omega-\tr_{\hat h}\alpha_t
                \ge n\left(\frac{\omega^n}{\hat\omega^n}\right)^{\frac1n}-na
                =ne^{\partial_t\psi/n}-na.
    \]
    Let $v:=\psi-n\int_0^t\log a(\tau)\,d\tau$.
    Then
    \begin{equation}
        \label{eqn reference trace inequality}
        \Delta_{\hat h}v\ge na(e^{\partial_t v/n}-1)
                \quad\text{and}\quad v(\cdot,0)=0.
    \end{equation}

    By the K\"ahler condition and the lower bound for $\Ric(\hat\omega)$
            outside a compact set, the underlying Riemannian metric $\hat h$
            has Ricci curvature bounded below.
    Moreover, $\Delta_{\hat h}$ differs from one half of the Riemannian
            Laplacian by a first-order term determined by a bounded vector
            field supported in that compact set.
    Fix $p\in M$.
    By Laplacian comparison, for each $R>0$, we may choose a cutoff function
            $0\le\eta_R\le1$, equal to $1$ on $B_{\hat h}(p,R)$ and vanishing
            outside $B_{\hat h}(p,2R)$, such that
    \[
        \Delta_{\hat h}\eta_R
                -2\frac{|\partial\eta_R|_{\hat h}^2}{\eta_R}
                \ge-C_0(R^{-2}+R^{-1})
                \quad\text{where}\quad\eta_R>0.
    \]
    Here $C_0>0$ is independent of $R$.
    In the maximum principle calculation below, we may treat $\eta_R$ as
            smooth by Calabi's trick.

    Consider
    \[
        F_R(x,t):=e^{-t}\eta_R(x)v(x,t)
                \quad\text{on}\quad M\times[0,S].
    \]
    Since $F_R$ has compact spatial support and $F_R(\cdot,0)=0$, any positive
            maximum is attained at a point $(x_R,t_R)$ with $t_R>0$.
    At this point, $v>0$ and $\partial_t v\ge v$, also when $t_R=S$ by
            taking the left time derivative.
    By \eqref{eqn reference trace inequality} and $a\ge1$, we obtain
    \[
        \Delta_{\hat h}v
                \ge na(e^{\partial_t v/n}-1)
                \ge\frac{a}{2n}(\partial_t v)^2
                \ge\frac{v^2}{2n}.
    \]
    By the spatial maximum condition and the estimate for $\eta_R$, we obtain
    \[
        \begin{aligned}
            0&\ge\eta_R\Delta_{\hat h}v+v\Delta_{\hat h}\eta_R
                    -2v\frac{|\partial\eta_R|_{\hat h}^2}{\eta_R} \\
            &\ge\frac{\eta_Rv^2}{2n}-C_0(R^{-2}+R^{-1})v.
        \end{aligned}
    \]
    Dividing by $v>0$, we obtain
    \[
        \max_{M\times[0,S]}F_R
                \le\eta_R(x_R)v(x_R,t_R)
                \le2nC_0(R^{-2}+R^{-1}).
    \]
    The same upper bound holds if $F_R\le0$.
    By letting $R\to+\infty$, we obtain $v\le0$ on $M\times[0,S]$.
    This proves the lemma.
\end{proof}

To obtain the lower bound for the potential, we use the positivity condition
        \eqref{eqn reference positivity} and completeness of the flow at
        positive times.
By combining the scalar curvature estimate with a trace inequality
        and applying Corollary~\ref{cor chern maximum principle},
        we obtain the desired lower bound.
No a priori bound on the potential or uniform equivalence
        of the evolving metrics to $\hat\omega$ is assumed.
See \cite{TW15JDG, HLT19CVPDE, LT20JDG, CL20TAMS, HLTT22CAG} on related lower
        bound estimates for the potential.

\begin{lemma}
    \label{lem lower potential}
    Under the assumptions of Theorem~\ref{thm main}, let $\omega(t)$ be an
            instantaneously complete solution to the Chern-Ricci flow on
            $M\times[0,S]$ with initial data $\omega_0$ and define $\psi$ by
            \eqref{eqn canonical potential}.
    Then, for all $x\in M$ and $0<t\le\min\{S,s\}$,
    \[
        \psi(x,t)\ge nt\left(\log\frac{\beta t}{s}-1\right)
                -\frac ts\osc\limits_M u.
    \]
\end{lemma}

\begin{proof}
    Let $h(t)$ be the Hermitian metric associated with $\omega(t)$.
    By replacing $u$ with $u-\sup\limits_M u$, we may assume that
            $-\osc\limits_M u\le u\le0$.
    For $0<t\le\min\{S,s\}$, let
    \[
        H:=\psi-\frac tsu-nt\left(\log\frac{\beta t}{s}-1\right).
    \]
    By \eqref{eqn potential time derivatives} and
            Corollary~\ref{cor scalar lower bound}, we obtain
            $\partial_t^2\psi\le n/t$.
    By integration by parts and noting that $\psi(\cdot,0)=0$, we obtain
    \begin{equation}
        \label{eqn potential time inequality}
        t\partial_tH-H=t\partial_t\psi-\psi-nt
                =\int_0^t\tau\partial_\tau^2\psi\,d\tau-nt\le0.
    \end{equation}
    In particular, $\partial_tH\le H/t<0$ wherever $H<0$.

    Since $t\le s$, by \eqref{eqn reference positivity}, we obtain
    \begin{equation*}
        \begin{aligned}
            \omega(t)-\sqrt{-1}\partial\bar\partial H
            &=\left(1-\frac ts\right)\omega_0
                    +\frac ts\bigl(\omega_0-s\Ric(\hat\omega)
                    +\sqrt{-1}\partial\bar\partial u\bigr) \\
            &\ge\frac{\beta t}{s}\hat\omega.
        \end{aligned}
    \end{equation*}
    By taking the $h(t)$-trace and applying the arithmetic-geometric mean
            inequality, we obtain
    \begin{equation}
        \label{eqn lower trace inequality}
        \begin{aligned}
            n-\Delta_hH
            &\ge\frac{\beta t}{s}\tr_h\hat\omega
                    \ge\frac{n\beta t}{s}e^{-\partial_t\psi/n} \\
            &=n\exp\left(-\frac{\partial_tH}{n}-\frac{u}{ns}\right)
                    \ge ne^{-\partial_tH/n},
        \end{aligned}
    \end{equation}
            where the last inequality follows from $u\le0$.
    By \eqref{eqn potential time inequality},
            \eqref{eqn lower trace inequality}, and
            $e^z\ge1+z+z^2/2$ for $z\ge0$, we obtain, wherever $H<0$,
    \begin{equation*}
        \Box_hH\ge\partial_tH+n(e^{-\partial_tH/n}-1)
                \ge\frac{(\partial_tH)^2}{2n}
                \ge\frac{H^2}{2nt^2}.
    \end{equation*}

    For any $0<\varepsilon<t\le\min\{S,s\}$, by applying
            Corollary~\ref{cor chern maximum principle} on $[\varepsilon,t]$
            with $K=+\infty$, we obtain
    \[
        H(x,t)\ge-\left(
                \int_\varepsilon^t\frac{d\tau}{2n\tau^2}\right)^{-1}
                =-\frac{2n\varepsilon t}{t-\varepsilon}.
    \]
    By letting $\varepsilon\to0^+$, we obtain $H\ge0$.
    Hence
    \[
        \psi\ge nt\left(\log\frac{\beta t}{s}-1\right)+\frac tsu
                \ge nt\left(\log\frac{\beta t}{s}-1\right)
                -\frac ts\osc\limits_M u.
    \]
    This completes the proof.
\end{proof}

\begin{proof}[Proof of Theorem~\ref{thm main}]
    Let $T:=\min\{S,s\}$, and let $h_i(t)$ and $\psi_i$ be the metrics
            and potentials associated with $\omega_i(t)$, respectively,
            where $\psi_i$ is defined by \eqref{eqn canonical potential}.
    By Lemmas~\ref{lem upper potential} and~\ref{lem lower potential},
            the functions $\psi_i$ are bounded on $M\times[0,T]$ and
            converge uniformly to zero as $t\to0^+$.

    By \eqref{eqn potential equation} and \eqref{eqn potential time derivatives},
            we obtain, for $i,j\in\{1,2\}$ with $i\ne j$,
    \[
        \Box_{h_i}(\psi_i-\psi_j)
        =\log\frac{\omega_i^n}{\omega_j^n}
                -n+\tr_{h_i}\omega_j\ge0.
    \]
    Let $\Psi:=\psi_1-\psi_2$.
    For $0<\varepsilon<T$, let
    \[
        a_\varepsilon
                :=\|\Psi(\cdot,\varepsilon)\|_{L^\infty(M)}.
    \]
    By applying Corollary~\ref{cor chern maximum principle}
            to $a_\varepsilon+\Psi$ and $a_\varepsilon-\Psi$
            on $M\times[\varepsilon,T]$, with respect to $h_1(t)$
            and $h_2(t)$, respectively, we obtain
    \[
        |\Psi(x,t)|\le a_\varepsilon
                \quad\text{for all }(x,t)\in M\times[\varepsilon,T].
    \]
    By letting $\varepsilon\to0^+$, we obtain $\Psi=0$.
    Hence $\omega_1=\omega_2$ on $M\times[0,T]$
            by \eqref{eqn potential equation}.
\end{proof}

\section{Consequences and examples}
\label{sec consequences}

In this section, we present applications of Theorem~\ref{thm main}.
We first consider complete initial metrics and use the initial metric as
        the reference metric.

\begin{proof}[Proof of Corollary~\ref{cor complete initial}]
    By applying Theorem~\ref{thm main} with $\hat\omega=\omega_0$, $C=1$,
            $K=K^-$, $u=0$, and $\beta=1-sK^+$ for any $s>0$ satisfying
            $sK^+<1$, we obtain the conclusion.
\end{proof}

\begin{corollary}
    Let $\omega_0$ be a complete Hermitian metric which is K\"ahler outside
            a compact set and satisfies $\Ric(\omega_0)=0$.
    Every smooth instantaneously complete Chern-Ricci flow from $\omega_0$ is
            stationary throughout its existence interval.
\end{corollary}

\begin{proof}
    By applying Corollary~\ref{cor complete initial} with $K^-=K^+=0$
            to the given solution and the complete stationary solution,
            we obtain the conclusion.
\end{proof}

In particular, the conclusion holds for every complete Ricci-flat K\"ahler
        initial metric. A direct proof of this special case is given in
        Appendix~\ref{sec ricci flat appendix}. More generally, if the initial metric is uniformly equivalent to a complete
        Ricci-flat K\"ahler metric, then the positivity condition holds for
        every $s>0$.

\begin{corollary}
    \label{cor ricci flat reference}
    Let $\hat\omega$ be a complete Ricci-flat K\"ahler metric on $M$.
    Let $\omega_0$ be a Hermitian metric which is K\"ahler outside a compact
            set, and suppose that
    \[
        C^{-1}\hat\omega\le\omega_0\le C\hat\omega,
    \]
            for some $C\ge1$.
    Then any two smooth instantaneously complete Chern-Ricci flows from
            $\omega_0$ agree throughout their common existence interval.
\end{corollary}

\begin{proof}
    By applying Theorem~\ref{thm main} with $K=0$, $u=0$, and
            $\beta=C^{-1}$ for every $s>0$, we obtain the conclusion.
\end{proof}

For example, this corollary applies to every smooth Hermitian metric on
        $\mathbb C^n$ whose associated form equals $\omega_{\mathrm E}$ outside
        a compact set.
By positivity on the compact set, we obtain the required uniform equivalence
        to $\omega_{\mathrm E}$.
The initial metric need not be K\"ahler on the compact set.

If the Chern-Ricci form of the reference metric is bounded above by a
        negative constant multiple of the reference metric, we can verify the
        positivity condition in Theorem~\ref{thm main} without a uniform
        lower bound for $\omega_0$ with respect to $\hat\omega$.
Thus the following corollary also applies to incomplete initial metrics.

\begin{corollary}
    \label{cor negative reference}
    Let $\hat\omega$ be a complete Hermitian metric on $M$, K\"ahler outside a
            compact set, such that
    \[
        -K_1\hat\omega\le\Ric(\hat\omega)\le-K_2\hat\omega
                \quad\text{for some}\quad K_1\ge K_2>0.
    \]
    Let $\omega_0$ be a Hermitian metric which is K\"ahler outside a compact
            set and satisfies $\omega_0\le C\hat\omega$ for some $C>0$.
    Then any two smooth instantaneously complete Chern-Ricci flows from
            $\omega_0$ agree throughout their common existence interval.
\end{corollary}

\begin{proof}
    By the upper bound for $\Ric(\hat\omega)$, we obtain
            \eqref{eqn reference positivity} with $u=0$ and $\beta=sK_2$
            for every $s>0$.
    Then we obtain the conclusion by applying Theorem~\ref{thm main}.
\end{proof}

We now apply Corollary~\ref{cor negative reference} to the unit ball
        to prove Corollary~\ref{cor unit ball}.

\begin{proof}[Proof of Corollary~\ref{cor unit ball}]
    The metric $\hat\omega$ is the standard  hyperbolic metric on
            $\mathbb B^n$.
    It is complete and has bounded curvature, with
            $\Ric(\hat\omega)=-(n+1)\hat\omega$.
    By applying Corollary~\ref{cor negative reference}, we obtain uniqueness.

    To prove existence with initial metric $\omega_{\mathrm E}$, let $\hat h$
            be the Hermitian metric associated with $\hat\omega$ and let
            $\rho:=-\log(1-|z|^2)$.
    The function $\rho$ is a proper exhaustion function, and we have
    \[
        \omega_{\mathrm E}\le\hat\omega,\qquad
        |\partial\rho|_{\hat h}^2\le1,\qquad
        \sqrt{-1}\partial\bar\partial\rho=\hat\omega.
    \]
    Moreover, both metrics are K\"ahler and
            $|\nabla^{\hat h}\omega_{\mathrm E}|_{\hat h}\le C_n$, where
            $C_n>0$ depends only on $n$.
    For every $s>0$, we have
    \[
        -\Ric(\hat\omega)+e^{-s}\bigl(\omega_{\mathrm E}+\Ric(\hat\omega)\bigr)
                \ge(n+1)(1-e^{-s})\hat\omega.
    \]
    By \cite[Theorems~1.1 and 1.3]{HLT19CVPDE}, we obtain an instantaneously
            complete solution $\tilde\omega(\tau)$ of
            $\partial_\tau\tilde\omega=-\Ric(\tilde\omega)-\tilde\omega$
            for all $\tau>0$.
    By \cite[Theorem~1.2]{HLT19CVPDE} and the flow equation, this solution
            extends smoothly to $\tau=0$ with initial metric
            $\omega_{\mathrm E}$.
    Finally, let
    \[
        \omega(t):=(1+t)\tilde\omega\bigl(\log(1+t)\bigr).
    \]
    Then $\omega(t)$ is a smooth instantaneously complete solution to
            \eqref{eqn chern flow} for all $t\ge0$ with initial metric
            $\omega_{\mathrm E}$.
\end{proof}

\section{Hermitian Ricci flows on complex surfaces}
\label{sec hermitian}

In this section, we study Riemannian Ricci flows on a complex surface
        $(M^2,J)$ that remain Hermitian for the fixed complex structure $J$.
The following proposition requires neither completeness nor curvature bounds.
We will apply it to prove Corollary~\ref{cor hermitian uniqueness}.
We consider the Riemannian Ricci flow
\begin{equation}
    \label{eqn real flow}
    \partial_t g=-2\Ric_g^{\mathbb R}
\end{equation}
        on a complex surface $(M^2,J)$, with $J$ fixed.
Throughout this section, $X,Y,Z$ denote smooth local real vector fields. We would like to mention again a complex surface $(M^2,J)$ is of real dimension $4$. 

\begin{proposition}
    \label{prop hermitian preservation}
    Let $(M^2,J)$ be a complex surface, and let $g(t)$, $t\in[0,S]$, be a
            smooth solution of \eqref{eqn real flow} that is Hermitian for $J$
            at every time.
    If $g(s)$ is K\"ahler for some $s\in[0,S]$, then $g(t)$ is K\"ahler for
            every $t\in[0,S]$.
\end{proposition}

Let $g$ be a Hermitian metric on $(M,J)$, and let
        $\omega(X,Y):=g(JX,Y)$ be its associated real $(1,1)$-form.
On a complex surface, $\alpha\mapsto\alpha\wedge\omega$ is an isomorphism from
        real $1$-forms to real $3$-forms.
Hence there is a unique real $1$-form $\theta$, called the Lee form, such that
\begin{equation}
    \label{eqn lee definition}
    d\omega=\theta\wedge\omega.
\end{equation}
The metric $g$ is K\"ahler if and only if $\theta=0$.

\begin{proof}[Proof of Proposition~\ref{prop hermitian preservation}]
    By differentiating $g(JX,JY)=g(X,Y)$ in time and using
            \eqref{eqn real flow}, we obtain that $\Ric_g^{\mathbb R}$
            is $J$-invariant.
    Let
    \[
        \rho(X,Y):=\Ric_g^{\mathbb R}(JX,Y)=-\Ric_g^{\mathbb R}(X,JY),
    \]
            so that $\rho$ is a real $(1,1)$-form and
            $\partial_t\omega=-2\rho$.
    Let $\mathcal R^{\mathbb R}$ denote the real scalar curvature of $g(t)$.
    For a real $1$-form $\alpha$, let
            $(\Ric^{\mathbb R}\alpha)(X):=\Ric^{\mathbb R}(\alpha^\sharp,X)$,
            with metric duals taken with respect to $g(t)$.
    We claim that
    \begin{equation}
        \label{eqn claim d rho}
        d\rho=\frac{\mathcal R^{\mathbb R}}{2}\theta\wedge\omega
                -\theta\wedge\rho
                =(\Ric^{\mathbb R}\theta)\wedge\omega.
    \end{equation}
    We prove both equalities directly, see also
            \cite[Lemma~4.1, (28)]{ADL26AGAG} for the first equality.

    Fix $t$ and let
    \[
        E:=\frac{\mathcal R^{\mathbb R}}{2}g-\Ric^{\mathbb R}.
    \]
    By the contracted Bianchi identity, we have $\operatorname{div}E=0$.
    Moreover, $E$ is $J$-invariant and $\tr_g E=\mathcal R^{\mathbb R}$.
    We use the orientation induced by $J$ and denote the Hodge star by $*$.
    Since $\rho$ is a real $(1,1)$-form and
            $\langle\rho,\omega\rangle=\mathcal R^{\mathbb R}/2$, we have
    \begin{equation}
        \label{eqn star rho}
        *\rho=\frac{\mathcal R^{\mathbb R}}{2}\omega-\rho
                =E(J\cdot,\cdot),
    \end{equation}
            see also \cite[Chapter~5, Theorem~3.16]{Wells08Book}.
    By the integrability of $J$ and \eqref{eqn lee definition}, we have
    \begin{equation}
        \label{eqn D J}
        2(D_XJ)Y=\theta(JY)X-\theta(Y)JX
                -g(X,JY)\theta^\sharp+g(X,Y)J\theta^\sharp.
    \end{equation}
    Let $\{e_i\}_{i=1}^4$ be a local orthonormal frame, and let
            $\delta\beta:=-*d*\beta$ for every real $2$-form $\beta$.
    By \eqref{eqn star rho}, \eqref{eqn D J}, and
            $\operatorname{div}E=0$, we obtain
    \begin{align*}
        \delta(*\rho)(X)
        &=\sum_{i=1}^4 E\bigl(e_i,(D_{e_i}J)X\bigr)\\
        &=\frac12(\tr_g E)\theta(JX)-E(\theta^\sharp,JX)\\
        &=\Ric^{\mathbb R}(\theta^\sharp,JX)
        =-\rho(\theta^\sharp,X).
    \end{align*}
    Since $\delta(*\rho)=-*d\rho$, we obtain
    \[
        *d\rho=\iota_{\theta^\sharp}\rho
                =*(\theta\wedge *\rho),
    \]
            where $\iota_X$ denotes contraction with $X$. For the second
            equality above, see also \cite[Proposition~8.2~(v)]{Viaclovsky07Notes}.
    By \eqref{eqn star rho}, we obtain
    \[
        d\rho=\theta\wedge *\rho
                =\frac{\mathcal R^{\mathbb R}}{2}\theta\wedge\omega
                    -\theta\wedge\rho.
    \]
    By the identity $*(\alpha\wedge\omega)=-\alpha\circ J$ for real
            $1$-forms $\alpha$, we also obtain
    \[
        *d\rho=\iota_{\theta^\sharp}\rho
                =-(\Ric^{\mathbb R}\theta)\circ J
                =*\bigl((\Ric^{\mathbb R}\theta)\wedge\omega\bigr).
    \]
    This proves \eqref{eqn claim d rho}.

    By differentiating \eqref{eqn lee definition} in time and using
            $\partial_t\omega=-2\rho$ and \eqref{eqn claim d rho}, we obtain
    \[
        (\partial_t\theta)\wedge\omega
        =-2d\rho+2\theta\wedge\rho
        =\bigl(\mathcal R^{\mathbb R}\theta-4\Ric^{\mathbb R}\theta\bigr)
                \wedge\omega.
    \]
    Since $\alpha\mapsto\alpha\wedge\omega$ is an isomorphism, we obtain
    \[
        \partial_t\theta=(\mathcal R^{\mathbb R}\,\id-4\Ric^{\mathbb R})\theta.
    \]
    At each fixed point of $M$, this is a homogeneous linear ordinary
            differential equation with smooth coefficients.
    Since $g(s)$ is K\"ahler, we have $\theta(s)=0$.
    By uniqueness for this equation, we obtain $\theta(t)=0$ for all
            $t\in[0,S]$.
    Hence $g(t)$ is K\"ahler for all $t\in[0,S]$.
\end{proof}

\begin{proof}[Proof of Corollary~\ref{cor hermitian uniqueness}]
    For $i=1,2$, let $\omega_i(t)$ be the associated real $(1,1)$-form of $g_i(t)$.
    By Proposition~\ref{prop hermitian preservation}, both flows are
            K\"ahler for all $t\in[0,S]$.
    Since the Riemannian Ricci form of a K\"ahler metric agrees with its
            Chern-Ricci form, we have $\partial_t\omega_i=-2\Ric(\omega_i)$.
    Let $\tau:=2t$ and
    \[
        \tilde\omega_i(\tau):=\omega_i(\tau/2),\qquad 0\le\tau\le2S.
    \]
    Then $\tilde\omega_1$ and $\tilde\omega_2$ are smooth instantaneously complete solutions
            of \eqref{eqn chern flow} with initial form $\omega_0$.
    By applying Theorem~\ref{thm main}, we obtain
            $\tilde\omega_1(\tau)=\tilde\omega_2(\tau)$ for all
            $\tau\in[0,\min\{2S,s\}]$.
    Therefore, we obtain the conclusion.
\end{proof}

\begin{remark}
    Corollary~\ref{cor hermitian uniqueness} has the following consequences.
    If the common initial K\"ahler metric $g_0$ is complete and $\omega_0$
            satisfies \eqref{eqn initial ricci bounds}, then by applying
            Corollary~\ref{cor complete initial} and using the same time
            change, we obtain $g_1(t)=g_2(t)$ at every common time $t$ with
            $2K^+t<1$.
    When $K^+>0$, equality at $t=1/(2K^+)$ follows by continuity if both flows
            are defined there.
    If $g_0$ is also Ricci-flat, then by comparing with the stationary
            solution, we obtain that every such flow is stationary throughout
            its existence interval.
    In particular, a smooth Ricci flow on $\mathbb R^4\cong\mathbb C^2$
            starting from the Euclidean metric, complete for positive times and
            Hermitian for the standard complex structure at every time, is
            stationary.
    More generally, by our recent improvement \cite[Corollary~1.4]{HP2610}, every smooth complete
            Ricci flow from a complete flat metric is stationary, without
            any Hermitian assumption on the evolving flow and any dimensional restriction. 
\end{remark}

\appendix
\section{The K\"ahler Ricci-flat  case}
\label{sec ricci flat appendix}

In this appendix, we provide a direct proof of stationarity for smooth
        instantaneously complete K\"ahler-Ricci flows starting from complete K\"ahler
        Ricci-flat  metrics.
Our study of strong uniqueness in this paper is motivated by this special case.

\begin{theorem}
    \label{thm ricci flat stationarity}
    Let $(M^n,\omega_0)$ be a  complete K\"ahler Ricci-flat  manifold.
    Let $\omega(t)$ be a smooth instantaneously complete K\"ahler-Ricci flow on
            $M\times[0,S]$ with $\omega(0)=\omega_0$.
    Then
    \[
        \omega(t)=\omega_0\quad\text{for all}\quad 0\le t\le S.
    \]
\end{theorem}

\begin{proof}
    Let $h(t)$ be the Hermitian metric associated with $\omega(t)$.
    Since $\Ric(\omega_0)=0$, by Corollary~\ref{cor scalar lower bound},
            we obtain $\mathcal R\ge0$.
    Let
    \[
        \psi(x,t):=\int_0^t\log
                \frac{\omega(\tau)^n}{\omega_0^n}(x)\,d\tau.
    \]
    By the flow equation, we have
    \[
        \omega(t)=\omega_0+\sqrt{-1}\partial\bar\partial\psi,
        \qquad
        \partial_t\psi=\log\frac{\omega(t)^n}{\omega_0^n},
        \qquad
        \partial_t^2\psi=-\mathcal R\le0.
    \]
    Since $\psi(\cdot,0)=\partial_t\psi(\cdot,0)=0$, by integrating in
            time, we obtain
    \[
        \partial_t\psi\le0,\qquad t\partial_t\psi(x,t)
                \le\int_0^t\partial_\tau\psi(x,\tau)\,d\tau
                =\psi(x,t)\le0.
    \]
    For $0<t\le S$, by taking the $h(t)$-trace and applying
            the arithmetic-geometric mean inequality, we obtain
    \begin{align*}
        \Box_h\psi
        &=\partial_t\psi-n+\tr_h\omega_0\\
        &\ge\partial_t\psi+n(e^{-\partial_t\psi/n}-1)\\
        &\ge\frac{(\partial_t\psi)^2}{2n}
                \ge\frac{\psi^2}{2nS^2},
    \end{align*}
    where we have used $e^x\ge1+x+x^2/2$ for $x\ge0$.
    By Proposition~\ref{prop global maximum principle} with $K=0$ and
            $\delta=1/(2nS^2)$, we obtain $\psi\ge0$.
    Hence $\psi=0$ and $\omega(t)=\omega_0$ for all $0\le t\le S$.
\end{proof}

\raggedbottom
\bibliographystyle{amsplain}
\bibliography{refs}

\end{document}

